\documentclass{zzz}
\usepackage{mathrsfs}
\usepackage{color}
\usepackage{tikz}
\usepackage{graphicx}
\usepackage{hyperref}

\hypersetup{
 colorlinks = true,
 urlcolor = blue!60!black,
 linkcolor = red!60!black,
 citecolor = green!60!black,
}

\newtheorem{thm}{Theorem}[section]

\makeatletter
\def\eqnarray{\stepcounter{equation}\let\@currentlabel=\theequation
\global\@eqnswtrue
\tabskip\@centering\let\\=\@eqncr
$$\halign to \displaywidth\bgroup\hfil\global\@eqcnt\z@
 $\displaystyle\tabskip\z@{##}$&\global\@eqcnt\@ne
 \hfil$\displaystyle{{}##{}}$\hfil
 &\global\@eqcnt\tw@ $\displaystyle{##}$\hfil
 \tabskip\@centering&\llap{##}\tabskip\z@\cr}

\def\endeqnarray{\@@eqncr\egroup
 \global\advance\c@equation\m@ne$$\global\@ignoretrue}

\def\@yeqncr{\@ifnextchar [{\@xeqncr}{\@xeqncr[5pt]}}
\makeatother

\begin{document}

\renewcommand{\PaperNumber}{***}

\FirstPageHeading

\ShortArticleName{Investigating Multidimensional Degenerate Hybrid Special Polynomials and Their Connection to Appell Sequences: Properties and Applications
}

\ArticleName{Investigating Multidimensional Degenerate Hybrid Special Polynomials and Their Connection to Appell Sequences: Properties and Applications
}

\newcommand{\orcidauthorC}{0000-0002-9545-7411}   
\Author{
Awatif Muflih Alqahtani $^{1,\dagger,\ddagger}$, Saleem Yousuf $^{2,\ddagger}$, Shahid Ahmad Wani$^{3}$ and Roberto S. Costas-Santos$^{4, \orcidC{}}$
}

\AuthorNameForHeading{
Alqahtani, A. M.; Yousuf, S.; Wani, S. A.; Costas-Santos, R. S.
}

\Address{%
$^{1}$ \quad Department of Mathematics,
Shaqra University, Riyadh, 
Saudi Arabia; aalqhtani@su.edu.sa\\
$^{2}$ \quad Department of Mathematics, National Institute of Technology, Srinagar, India; saleemamu12@gmail.com\\
$^{3}$ \quad Symbiosis Institute of Technology, Symbiosis International (Deemed University), Pune, India; shahidwani177@gmail.com~shahid.wani@sitpune.edu.in\\
$^{4}$ \quad Department of Quantitative 
Methods, Universidad Loyola Andaluc\'ia, 
E-41704 Seville, Spain; rscosa@gmail.com}

\ArticleDates{Received ???, in final form ????; Published online ????}

\Abstract{
This paper explores the operational principles and monomiality principle that significantly shape the development of various special polynomial families. We argue that applying the monomiality principle yields novel results while remaining consistent with established findings. The primary focus of this study is the introduction of degenerate multidimensional Hermite-based Appell polynomials (DMHAP), denoted as ${}_{\mathbb{H}}\mathbb{A}^{[r]}_n(l_1, l_2, l_3, \dots, l_r; \vartheta)$. These DMHAP form essential families of orthogonal polynomials, demonstrating strong connections with classical Hermite and Appell polynomials. Additionally, we derive symmetric identities and examine the fundamental properties of these polynomials. Finally, we establish an operational framework to investigate and develop these polynomials further.
}


\Classification{33E20; 33C45; 33B10; 33E30; 11T23}
\section{Introduction}\label{sec:intro}
Special functions and polynomials are indispensable in various branches of mathematics, physics, engineering, and scientific disciplines, offering efficient solutions to problems where elementary functions are insufficient. Special functions, such as Bessel, hypergeometric, and Legendre functions, naturally emerge in solving differential equations and modeling physical phenomena. For instance, Bessel functions describe wave propagation in cylindrical systems, while Legendre polynomials are central to potential theory and celestial mechanics. These functions also play a critical role in quantum mechanics, heat conduction, and electromagnetic theory, providing analytical frameworks for complex systems with symmetry and boundary conditions.  

These special polynomials, as fundamental mathematical tools, are essential for approximation, interpolation, and modeling relationships. Orthogonal polynomials, like Chebyshev and Jacobi polynomials, are crucial in numerical methods, minimizing errors in approximation and integration. Beyond mathematics, their applications extend to engineering, where they aid in control systems and signal processing, and to data science, where they support regression and kernel-based methods. Together, special functions and polynomials bridge theoretical studies and practical applications, enabling advances in diverse fields by providing versatile and powerful solutions to complex challenges.  Their significance extends to diverse areas, including numerical analysis, signal processing, and computer science. Special functions and polynomials serve as indispensable mathematical tools, facilitating the description and analysis of complex systems and phenomena across various scientific and engineering disciplines. Due to the significance of these special polynomial sequences, 
Paul Appell developed the Appell polynomials in the 19th century \cite{Appell}. They find applications across various fields of mathematics and physics including quantum mechanics, study of differential equations, and algebraic geometry. They show relationships with hypergeometric functions and Jacobi polynomials, among other special function families. 

Appell polynomials \cite{Appell} are an important class of polynomial sequences distinguished by their unique recursive properties. They play a vital role in calculus, combinatorics, and mathematical analysis because of their close connections with differentiation and integration. These polynomials generalize several well-known families, such as Bernoulli, Euler, and Hermite polynomials, making them versatile tools in diverse mathematical studies. Appell polynomials are widely applied in number theory, approximation methods, and differential equation solutions, showcasing their significance across various branches of mathematics and related disciplines. The ``Appell polynomial family $\mathbb{A}_{k}(l_1)$", as given by Appell in the 19th century, satisfies the differential equation:
\begin{equation}\label{1.2}
	\frac{d}{d{l_1}}\mathbb{A}_{k}(l_1)=k~
	\mathbb{A}_{k-1}(l_1),~~~~~~~~~~k\in \mathbb{N}_{0}
\end{equation}
and the generating relation expression:
\begin{equation}\label{1.3}
		\mathbb{A}(\xi)~e^{l_1\xi}=\sum _{k=0}^{\infty }\mathbb{A}_{k}(l_1) \frac{\xi^{k}}{k!},
\end{equation}
with $\mathbb{A}(\xi)$ represented as a power series:
\begin{equation}\label{1.4}
	\mathbb{A}(\xi)=\sum\limits_{k=0}^{\infty}\mathbb{A}_k
	\frac{\xi^k}{k!},~~~~~~\mathbb{A}_0 \neq 0.
\end{equation}

Named in honour of Paul Appell, these polynomials were initially introduced in his exploration of elliptic functions. The Appell polynomials demonstrate intriguing group characteristics, constituting an abelian group through composition. This group structure emerges from combining two Appell polynomials resulting in another Appell polynomial. The generative relationship \eqref{1.3} enables the representation of the exponential operator $e^{l_1 \xi}$ limitless series involving these polynomials, simplifying specific integrals and computations. Appell polynomials are employed in mathematical physics, notably in investigating quantum mechanics, electromagnetism, and fluid dynamics. Their adaptability extends to the resolution of differential equations, the exploration of integrable systems, and the derivation of recursion relations for coefficients, rendering them a valuable instrument in diverse mathematical and physical scenarios.\\

Recently, there has been significant interest in the development and study of hybrid special polynomials due to their remarkable potential in addressing complex mathematical and scientific problems. These polynomials are constructed by combining the unique properties of distinct families of classical polynomials, resulting in an enriched framework with enhanced capabilities. By leveraging features such as orthogonality, recurrence relations, and generating functions, hybrid special polynomials offer powerful tools for solving differential equations, modeling intricate physical systems, and advancing numerical analysis techniques.  

The versatility of hybrid special polynomials extends across various fields, including quantum mechanics, where they help analyze wave functions and energy levels, and signal processing, where they support filter design and spectral analysis. In approximation theory, their adaptability enables precise modelling of complex functions. Additionally, hybrid polynomials provide new perspectives on algebraic and combinatorial structures, fostering deeper insights into mathematical theory. Their ability to bridge multiple mathematical domains makes hybrid special polynomials an essential and innovative tool in contemporary research, unlocking novel solutions to multidisciplinary challenges, see for example \cite{SubHAP,SubdetSAP,AraciWani,ramirez2}.

The introduction of multidimensional special polynomials and hybrid special polynomials has significantly enriched the mathematical landscape, particularly through the development of multidimensional, hybrid, multi-indexed, degenerate, and other polynomial sequences. Those sequences have in-depth applications in various disciplines of mathematics such as algebraic combinatorics and inumerative combinatorics.  In approximation theory and physics, classical polynomials such as Laguerre, Chebyshev, Legendre, and Jacobi polynomials play a pivotal role in solving ordinary differential equations and modelling complex systems. The multidimensional nature of these polynomials enables them to address higher-dimensional problems, enhancing their applicability across scientific fields.  

Among these sequences, Hermite polynomials \cite{Hermite} stand out as particularly versatile and impactful. Originally introduced by Hermite, these polynomials have found applications across computer science, physics, and engineering. They are integral to quantum mechanics, providing essential tools for understanding the behaviour of the hydrogen atom and characterizing harmonic oscillator wave functions. Moreover, Hermite polynomials are fundamental in the quantum theory of light, where they model photon behaviour and energy distribution. Beyond physics, these polynomials are also employed in probability theory to solve problems involving normal distributions and in kinetic theory to analyze the energy distribution of gas molecules, demonstrating their utility across diverse areas of mathematics and science.  

Degenerate special functions further extend the reach of classical polynomials, offering valuable insights into systems with inherent degeneracies or symmetries. These functions are particularly significant in quantum mechanics, where they describe the energy levels and wavefunctions of quantum harmonic oscillators and other systems. By providing solutions to the Schrödinger equation, degenerate special functions contribute to our understanding of particle behaviour at the quantum level, making them indispensable in quantum physics. Their utility extends to other fields as well, where they illuminate complex processes governed by symmetries, further underscoring their importance in advancing both theoretical and applied sciences. Furthermore, degenerate special functions are more useful in mathematical physics because of the operational formalism that goes along with them. To help scholars and practitioners in their endeavours, this formalism offers a methodical way to manipulate these functions and solve mathematical issues. The practical approaches provide efficient computations and the extraction of novel findings, advancing applied mathematics and theoretical physics. Degenerate special functions are used in approximation theory in addition to differential equations. These useful tools help with numerical analyses and calculations by approximating complicated mathematical expressions with simpler polynomial forms. Applying degenerate special functions in approximation theory aids the creation of effective algorithms and computing techniques.

Many writers have expressed interest in investigating and revealing different properties of degenerate special polynomials, as demonstrated by several publications, including \cite{RCS,Kyun3VDHP,HKW,DJ1,DJ2,DJ3}. Wani and collaborators have recently developed a range of specialized doped polynomials, highlighting their distinct properties and dynamic behaviors. These advancements hold significant relevance in engineering applications, as they provide insights into solving complex problems and optimizing system performance. For further details and specific examples, refer to \cite{w1,WaniHAP,Ata6,w2}. The study of degenerate special functions extends beyond classical scenarios, including recent research on ``determinant forms, operational formalism, approximation properties, and generating expressions". These investigations deepen our understanding of these special functions' intricate relationships and properties, paving the way for novel applications and theoretical advancements.

The most significant concept of ``Poweriods"  given by ``Steffenson in 1941" \cite{Stefpoweriod}, focuses on monomiality, a notion further refined by Dattoli \cite{DatHBLB,GRCV}. This concept has evolved into a crucial framework for analyzing exceptional polynomials. Monomiality involves representing a polynomial set using monomials, which serve as the fundamental components of polynomials. This method deepens the understanding of the structure and characteristics of the polynomial set.

For the polynomial set ${\sigma_k(l_1)}{k \in \mathbb{N}}$, the ``operators $\hat{\mathcal{M}}$ and $\hat{\mathcal{D}}$ serve as multiplicative and derivative operators", respectively, and are crucial in the analysis of special polynomials. The formula representing property of $\hat{\mathcal{M}}$ is as follows:
\begin{equation}\label{1.16}	\sigma_{k+1}(l_1)=\hat{\mathcal{M}}\{\sigma_k(l_1)\},
\end{equation}
by taking the previous polynomial $\sigma_k(l_1)$ and creates a new one $\sigma_{k+1}(l_1)$. Similarly, derivative feature of $\hat{\mathcal{D}}$ is expressed as:
\begin{equation}\label{1.17}
	k~\sigma_{k-1}(l_1)=\hat{\mathcal{D}}\{\sigma_k(l_1)\}.
\end{equation}
The multiplicative and derivative properties of these operators are illustrated through formulas \eqref{1.16} and \eqref{1.17}, which facilitate the development of novel polynomials through the transformation of existing ones. The integration of monomiality principles with operational techniques has advanced the formation of doped special polynomials, distinguished by their unique features and versatile applicability to diverse mathematical challenges. The study of these hybrid special polynomials continues to be a dynamic and evolving area of research. Their versatile nature has led to impactful applications across diverse disciplines, including ``computational algorithms in computer science, advanced modeling techniques in engineering, and theoretical investigations in physics".

The theory of ``quasi-monomials and Weyl groups", embodied by the set of polynomials $\{\sigma_k(l_1)\}_{m \in \mathbb{N}}$ and operators delineated by expressions \eqref{1.16} and \eqref{1.17}, stands as a potent framework for analyzing specific function and satisfies the axioms
\begin{equation}\label{1.18}
	[\hat{\mathcal{D}},\hat{\mathcal{M}}]=
	\hat{\mathcal{D}}\hat{\mathcal{M}}-
	\hat{\mathcal{M}}\hat{\mathcal{D}}=
	\hat{1},
\end{equation}
and
\begin{equation}\label{1.19}
	\hat{\mathcal{M}}\hat{\mathcal{D}}
	\{\sigma_k(l_1)\}=k~\sigma_k(l_1).
\end{equation}
 This polynomial set forms a Weyl group, adhering to the axioms outlined in \eqref{1.18}. The quasi-monomials satisfy a differential equation \eqref{1.19}, establishing their status as ``eigenfunctions of the operator $\hat{\mathcal{M}}\hat{\mathcal{D}}$ with an eigenvalue of $k$". The explicit solution to this differential equation is presented as
 \begin{equation}\label{1.20}
 	\sigma_k(l_1)=\hat{\mathcal{M}}^k~\{1\},
 \end{equation} 
 with the initial condition $\sigma_0(l_1)=1$, thus permits a recursive computation of the quasi-monomials, providing a robust tool for their systematic analysis. 	
 
The generating relation 
\begin{equation}\label{1.21}
	e^{\xi \hat{\mathcal{M}}}\{1\}=
	\sum\limits_{k=0}^\infty \sigma_k(l_1)\frac{\xi^k}{k!},~~~|\xi|<\infty~~,
\end{equation}
presents quasi-monomials in the form of power series, enabling precise calculations and revealing diverse applications in exploring the characteristics of special functions. Consequently, the theoretical structure of quasi-monomials and their connection with Weyl groups emerges as a robust framework for examining various families of functions, offering extensive utility across mathematical and physical sciences.

Furthermore, as researchers delve into the theory of ``quasi-monomials and Weyl groups" associated with degenerate special functions, new avenues for exploration open up. The quasi-monomials, acted upon by multiplicative and derivative operators, form a Weyl group, revealing rich algebraic structures. This framework provides a powerful toolkit for constructing and studying new families of special functions, expanding the scope of their applications in mathematical and physical contexts.

Understanding degenerate forms of special functions is crucial for grasping the mathematical principles underlying various physical phenomena. These polynomials play a key role in describing the behavior of quantum systems, including the quantum harmonic oscillator. They are widely used in many branches of science and engineering and are valued for their importance in theoretical and applied mathematics. More discoveries and innovations are expected as the theory of special functions progresses.

Recently, ``two and three variable degenerate variants of Hermite polynomials" were presented by Ryoo and Hwang \cite{HKW,Kyun3VDHP}:
\begin{equation}\label{2}
	\sum _{k=0}^{\infty }\mathcal{J}_{k}(l_1,l_2;\kappa)\frac{\xi^{k} }{k!}=(1+\kappa)^{\xi(\frac{l_1+l_2\xi}{\kappa})},
\end{equation}
and
\begin{equation}\label{3}
	\sum _{k=0}^{\infty }\mathcal{F}_{k}(l_1,l_2,l_3;\kappa)\frac{\xi^{k} }{k!}=(1+\kappa)^{\xi(\frac{l_1+l_2\xi+l_3\xi^2}{\kappa})}.
\end{equation}

Special function theory advancements have led to the creation of multi-dimensional special functions with multiple dimensions and indices. These functions, leveraging numerous variables and multiple indices, offer a more adaptable and effective approach to solving intricate problems. Multi-variable special functions find applications in quantum mechanics, statistical physics, fluid dynamics, engineering, computer science, and finance.

Given the significance of degenerate special polynomials in mathematical physics, we introduce a complex set of multidimensional degenerate Hermite-based Appell polynomials. These polynomials play a crucial role in studying partition functions of classical and quantum mechanical systems, expressing them in terms of integrals of Gaussian functions. They also contribute to constructing vertex operators generating the Virasoro algebra, essential in studying conformal field theories. Furthermore, these polynomials describe the eigenvalues of random matrices, aiding in computing moments and probability density functions.

The manuscript is structured into well defined sections, each addressing crucial aspects of the ``multidimensional degenerate Hermite-based Appell polynomials" and their properties. Section 1 focuses on foundational elements, beginning with the derivation and formal proof of the generating function for these DMHAP, denoted as ${}_{\mathbb{H}}\mathbb{A}^{[r]}_n(l_1, l_2, l_3, \dots, l_r; \vartheta)$. This generating relation act as a cornerstone for exploring their structural and analytical characteristics. Additionally, the section delves into the construction of ``multiplicative and derivative operators", culminating in the formulation of differential equations that govern these polynomials. These developments provide a robust framework for understanding the behavior and applications of multidimensional degenerate Hermite-based Appell polynomials in various mathematical contexts. Section 2 builds upon the foundational work by exploring the symmetrical properties and operational connections of these polynomials. Symmetric identities, a vital aspect of their algebraic structure, are rigorously established, highlighting their intrinsic mathematical elegance. Furthermore, the operational connection is thoroughly investigated, offering insights into how these polynomials interact within broader mathematical frameworks. These findings not only enhance the theoretical understanding of the polynomials but also pave the way for their application in solving complex mathematical problems. 
The collective contributions of these sections provide a comprehensive examination of degenerate multidimensional Hermite-based Appell polynomials, enriching the mathematical literature and advancing knowledge in the field. This work has the potential to influence diverse areas of mathematics, including combinatorics, approximation theory, and the study of special functions, thereby underscoring its significance in modern mathematical research.	
\section{Degenerate multidimensional Hermite based Appell polynomials}\label{sec:deg}
This section introduces and investigates a novel class of degenerate multidimensional Hermite-based Appell polynomials (DMHAP), uncovering their fundamental properties and structure. These polynomials are constructed by generating relation that captures their intricate dependencies and algebraic richness. The primary generating function, given by 
\begin{equation}\label{2.1}
	\sum_{n=0}^{\infty} {}_{\mathbb{H}}\mathbb{A}^{[r]}_n(l_1,l_2,l_3,\cdots,l_r;\kappa)\frac{\xi^n}{n!}=\mathbb{A}(\xi)(1+\kappa)^{\xi\Big(\frac{l_1+l_2\xi+l_3\xi^2+\cdots+l_n\xi^{n-1}}{\kappa}\Big)},
\end{equation}
establishes the direct connection between the polynomials and their generating parameters, including the degenerate parameter \(\kappa\) and the multi-dimensional coefficients \(l_1, l_2, \dots, l_r\). This form emphasizes the compact and recursive nature of the polynomials, enabling systematic analysis and derivation of their explicit expressions. 

An equivalent expression for the generating function is given by  
\begin{equation}\label{2.2}
	\sum_{n=0}^{\infty} {}_{\mathbb{H}}\mathbb{A}^{[r]}_n(l_1,l_2,l_3,\cdots,l_r;\kappa)\frac{\xi^n}{n!}=\mathbb{A}(\xi)(1+\kappa)^{\frac{l_1\xi}{\kappa}} (1+\kappa)^{\frac{l_2\xi^2}{\kappa}}(1+\kappa)^{\frac{l_3\xi^3}{\kappa}}\cdots(1+\kappa)^{\frac{l_r\xi^r}{\kappa}}.
\end{equation}
This alternative formulation expands the generating function into a product form, illustrating how the polynomial sequence is influenced by the multidimensional structure of the coefficients \(l_1, l_2, \dots, l_r\). The factorization highlights the modular construction of the generating function, which is particularly useful for analyzing and manipulating the polynomials in applied settings.  

These generating relations define the DMHAP and provide a foundation for deriving their properties, such as recurrence relations, operational representations, and differential equations. Through these expressions, the polynomials' dependence on the degenerate parameter \(\kappa\) is made explicit, enabling a deeper exploration of their applications in approximation theory, quantum mechanics, and other areas of mathematics and physics. The results presented here establish a robust framework for understanding and utilizing these polynomials in both theoretical and practical contexts.
	
\begin{thm}
For DMHAP, denoted by ${}{}_{\mathbb{H}}\mathbb{A}^{[r]}_n(l_1,l_2,l_3,\cdots,l_r;\kappa)$, the expression is found by using the power series expansion of the product  $\mathbb{A}(\xi)$ and\\ $(1+\kappa)^{\frac{l_1\xi}{\kappa}} (1+\kappa)^{\frac{l_2\xi^2}{\kappa}}(1+\kappa)^{\frac{l_3\xi^3}{\kappa}}\cdots(1+\kappa)^{\frac{l_r\xi^r}{\kappa}}$. 
In other words,
\begin{multline}\label{2.4}
{}{}_{\mathbb{H}}\mathbb{A}^{[r]}_0(l_1,l_2,l_3,\cdots,l_r;\kappa)+		
{}{}_{\mathbb{H}}\mathbb{A}^{[r]}_1(l_1,l_2,l_3,\cdots,l_r;\kappa)~\frac{\xi}{1!}+{}{}_{\mathbb{H}}\mathbb{A}^{[r]}_2(l_1,l_2,l_3,\cdots,l_r;\kappa)\frac{\xi^2}{2!}
		\\
		+\cdots+{}{}_{\mathbb{H}}\mathbb{A}^{[r]}_n(l_1,l_2,l_3,\cdots,l_r;\kappa)\frac{\xi^m}{m!}+ \cdots=\mathbb{A}(\xi)(1+\kappa)^{\frac{l_1\xi}{\kappa}} (1+\kappa)^{\frac{l_2\xi^2}{\kappa}}(1+\kappa)^{\frac{l_3\xi^3}{\kappa}}\cdots(1+\kappa)^{\frac{l_r\xi^r}{\kappa}}.
	\end{multline}
\end{thm}

\begin{proof}
Expanding the expression $\mathbb{A}(\xi)(1+\kappa)^{\frac{l_1\xi}{\kappa}} (1+\kappa)^{\frac{l_2\xi^2}
	{\kappa}}(1+\kappa)^
{\frac{l_3\xi^3}{\kappa}}
\cdots(1+\kappa)^{\frac{l_r\xi^r}
	{\kappa}}$ utilising the Newton series for finite differences, arranging the product of the function's $\mathbb{A}(\xi)~(1+\kappa)^{\frac{l_1\xi}{\kappa}} (1+\kappa)^{\frac{l_2\xi^2}{\kappa}}
	(1+\kappa)^{\frac{l_3\xi^3}
	{\kappa}}\cdots(1+\kappa)^
	{\frac{l_r\xi^r}{\kappa}}$  at $l_1=l_2=\cdots~l_r=0$
	  with respect to the powers of $\xi$, we encounter the special polynomials ${}{}_{\mathbb{H}}\mathbb{A}^{[r]}_n(l_1,l_2,l_3,\cdots,l_r;\kappa)$. These polynomials are the generating function for DMHAP and are stated in equation \eqref{2.4} as coefficients of $\frac{\xi^m}{m!}$.	  
	\end{proof}
	
Furthermore, we provide the following results to determine the multiplicative and derivative operators for DMHAP ${}{}_{\mathbb{H}}\mathbb{A}^{[r]}_n(l_1,l_2,l_3,...,l_r;\kappa)$.
	
\begin{thm}
The below listed  operators for the  DMHAP ${}{}_{\mathbb{H}}\mathbb{A}^{[r]}_n(l_1,l_2,l_3,\cdots,l_r;\kappa)$ usually called as ``multiplicative and derivative operators" holds true:
\begin{multline}\label{2.7}	
\hat{{\mathbb{M}_{{}_{\mathbb{H}}\mathbb{A}^{[r]}}}}=\Bigg(\Big(\frac{l_1}{\kappa}+\frac{\mathbb{A}'(\xi)}{\mathbb{A}(\xi)}\Big){{log(1+\kappa)}}+2 l_2\frac{\partial}{\partial l_1}+3 l_3\Big(\frac{\kappa}{log(1+\kappa)}\Big)\frac{\partial^2}
{\partial l_1^2}+\cdots
\\
+r l_r\Big(\frac{\kappa}{log(1+\kappa)}\Big)^{r-1}\frac{\partial^{r-1}}
{\partial l_1^{r-1}}\Bigg),
\end{multline}
and
\begin{equation}\label{2.8}
\hat{{\mathbb{D}_{{}_{\mathbb{H}}\mathbb{A}^{[r]}}}}=\frac{\kappa}{log(1+\kappa)}~D_{l_1}.
\end{equation} 
\end{thm}
	
\begin{proof}
By taking a partial derivative of~\eqref{2.2} with respect to $l_1$, we obtain
\begin{multline}\label{2.9}
\frac{\xi}{\kappa}~log(1+\kappa)\Bigg\{\mathbb{A}(\xi)(1+\kappa)^{\frac{l_1\xi}{\kappa}} (1+\kappa)^{\frac{l_2\xi^2}{\kappa}}(1+\kappa)^{\frac{l_3\xi^3}{\kappa}}\cdots(1+\kappa)^{\frac{l_r\xi^r}{\kappa}}
\Bigg\}\\
=\frac{\partial}{{\partial {l_1}}}\Bigg\{\mathbb{A}(\xi)(1+\kappa)^{\frac{l_1\xi}{\kappa}} (1+\kappa)^{\frac{l_2\xi^2}{\kappa}}(1+\kappa)^{\frac{l_3\xi^3}{\kappa}}\cdots(1+\kappa)^{\frac{l_r\xi^r}{\kappa}}\Bigg\}.
\end{multline}
As a result, \eqref{2.9} may be expressed as identity: 
\begin{multline}\label{2.10}
\xi\Bigg\{\mathbb{A}(\xi)(1+\kappa)^{\frac{l_1\xi}{\kappa}} (1+\kappa)^{\frac{l_2\xi^2}{\kappa}}
(1+\kappa)^{\frac{l_3\xi^3}{\kappa}}\cdots(1+\kappa)^{\frac{l_r\xi^r}{\kappa}}\Bigg\}\\
=
\frac{\kappa}{log(1+\kappa)}~	\frac{\partial}{{\partial {l_1}}}\Bigg\{\mathbb{A}(\xi)(1+\kappa)^{\frac{l_1\xi}{\kappa}} (1+\kappa)^{\frac{l_2\xi^2}{\kappa}}
(1+\kappa)^{\frac{l_3\xi^3}{\kappa}}\cdots(1+\kappa)^{\frac{l_r\xi^r}{\kappa}}\Bigg\},
\end{multline}
and streamlined in the form
\begin{equation}\label{2.11}
\begin{array}{cc}	
\xi\Bigg\{\sum\limits_{n=0}^{\infty}{}{}_{\mathbb{H}}\mathbb{A}^{[r]}_n(l_1,l_2,l_3,\cdots,l_r;\kappa)\frac{\xi^n}{n!}\Bigg\}=\frac{\kappa}{log(1+\kappa)}\frac{\partial}{{\partial {l_1}}}\Bigg\{\sum\limits_{n=0}^{\infty}{}{}_{\mathbb{H}}\mathbb{A}^{[r]}_n(l_1,l_2,l_3,\cdots,l_r;\kappa)\frac{\xi^n}{n!}\Bigg\}.
\end{array}
\end{equation}
After partial differentiation of \eqref{2.2} with respect to $\xi$ on both sides, it can be shown that
\begin{align}\label{2.12}
\frac{\partial}{{\partial {\xi}}}\Bigg\{\mathbb{A}(\xi)(1+\kappa)^{\frac{l_1\xi}{\kappa}} (1+\kappa)^{\frac{l_2\xi^2}{\kappa}}
(1+\kappa)^{\frac{l_3\xi^3}{\kappa}}\cdots(1+\kappa)^{\frac{l_r\xi^r}{\kappa}}\Bigg\}=\\
\frac{\partial}{{\partial {\xi}}}\Bigg\{\sum\limits_{n=0}^{\infty}{}{}_{\mathbb{H}}\mathbb{A}^{[r]}_n(l_1,l_2,l_3,\cdots,l_r;\kappa)\frac{\xi^{n}}{n!}\Bigg\}
\end{align}
and further expressed as
\begin{multline}\label{2.13}
\Bigg(\frac{\mathbb{A}'(\xi)}{\mathbb{A}(\xi)}+\frac{l_1}{\kappa}+2\frac{l_2}{\kappa}\xi+3\frac{l_3}{\kappa}\xi^2+\cdots+r\frac{l_r}{\kappa} \xi^{r-1}\Bigg)\Bigg\{\mathbb{A}(\xi)(1+\kappa)^{\frac{l_1\xi}{\kappa}} (1+\kappa)^{\frac{l_2\xi^2}{\kappa}}
(1+\kappa)^{\frac{l_3\xi^3}{\kappa}}
\\
\times \cdots(1+\kappa)^{\frac{l_r\xi^r}{\kappa}}\Bigg\}=\sum\limits_{n=0}^{\infty}~~n{}{}_{\mathbb{H}}\mathbb{A}^{[r]}_n(l_1,l_2,l_3,\cdots,l_r;\kappa)\frac{\xi^{n-1}}{n!}.
\end{multline}
Thus, expression \eqref{2.13} can be stated in light of expression \eqref{2.11} as follows
\begin{multline}\label{2.14}
\Bigg(\frac{\mathbb{A}'(\xi)}{\mathbb{A}(\xi)}+\frac{l_1}{\kappa}+2\frac{l_2}{log(1+\kappa)}\frac{\partial}{\partial{l_1}}+3{l_3}\frac{\kappa}{(log(1+\kappa))^2}\frac{\partial^2}{\partial {l_1}^2}+\cdots+r{l_r}\frac{\kappa^{r-2}}{(log(1+\kappa))^{r-1}} \Bigg)\\
 \Bigg\{\mathbb{A}(\xi)(1+\kappa)^{\frac{l_1\xi}{\kappa}}(1+\kappa)^{\frac{l_2\xi^2}{\kappa}}(1+\kappa)^{\frac{l_3\xi^3}{\kappa}}
\cdots(1+\kappa)^{\frac{l_r\xi^r}{\kappa}}\Bigg\}
=\sum\limits_{n=0}^{\infty}~~n~~{}{}_{\mathbb{H}}\mathbb{A}^{[r]}_n(l_1,l_2,l_3,\cdots,l_r;\kappa)
\frac{\xi^{n-1}}{n!}.
\end{multline}
By simplifying and substituting the right-hand side of \eqref{2.2} with that of the previous equation, we obtain:
\small{\begin{multline}\label{2.15}
\Bigg(\Big(\frac{l_1}{\kappa}+\frac{\mathbb{A}'(\xi)}{\mathbb{A}(\xi)}\Big){{log(1+\kappa)}}+2{l_2}\frac{\partial}{\partial{l_1}}+3{l_3}\Big(\frac{\kappa}{log(1+\kappa)}\Big)\frac{\partial^2}{\partial {l_1}^2}+\cdots+r{l_r}\Big(\frac{\kappa}{log(1+\kappa)}\Big)^{r-2}\frac{\partial^{r-1}}{\partial {l_1}^{r-1}} \Bigg)
\\
\times~\Bigg\{\sum\limits_{n=0}^{\infty}{}{}_{\mathbb{H}}\mathbb{A}^{[r]}_n(l_1,l_2,l_3,\cdots,l_r;\kappa)\frac{\xi^{n}}{n!}\Bigg\}
=\sum\limits_{n=0}^{\infty}~~n~~{}{}_{\mathbb{H}}\mathbb{A}^{[r]}_n(l_1,l_2,l_3,\cdots,l_r;\kappa)\frac{\xi^{n-1}}{n!}.
\end{multline}}
Equation \eqref{2.7} is obtained by substituting $n+1$ for $n$ on the right-hand side of equation \eqref{2.15} and then comparing the coefficients of $\frac{\xi^n}{n!}$.

Additionally, considering identity expression \eqref{2.11}, we discover
\begin{equation}\label{2.16}
\begin{array}{cc}	
\frac{\kappa}{log(1+\kappa)}~		\frac{\partial}{{\partial {l_1}}}\Bigg\{\sum\limits_{n=0}^{\infty}{}{}_{\mathbb{H}}\mathbb{A}^{[r]}_n(l_1,l_2,l_3,\cdots,l_r;\kappa)\frac{\xi^n}{n!}\Bigg\}=
\Bigg\{\sum\limits_{n=0}^{\infty}{}{}_{\mathbb{H}}\mathbb{A}^{[r]}_n(l_1,l_2,l_3,\cdots,l_r;\kappa)\frac{\xi^{n+1}}{n!}\Bigg\}.
\end{array}
\end{equation}
To obtain assertion \eqref{2.8}, we substitute $n-1$ for $n$ in the right-hand side of the previous expression \eqref{2.16}, and then compare the exponents of $\frac{\xi^{n}}{n!}$.
\end{proof}

Next, we show that DMHAP ${}{}_{\mathbb{H}}\mathbb{A}^{[r]}_n(l_1,l_2,l_3,\cdots,l_r;\kappa)$ satisfies the differential equation by presenting the following result:
	
\begin{thm}
The DMHAP ${}{}_{\mathbb{H}}\mathbb{A}^{[r]}_n(l_1,l_2,l_3,\cdots,l_r;\kappa)$ fulfils the differential expression:
\begin{multline}\label{2.17}
\Bigg(\Big(\frac{l_1}{\kappa}+\frac{\mathbb{A}'(\xi)}{\mathbb{A}(\xi)}\Big){{log(1+\kappa)}}\frac{\partial}{\partial l_1}+2 l_2\frac{\partial^2}{\partial{l_1}^{2}}+3 l_3\Big(\frac{\kappa}{log(1+\kappa)}\Big)\frac{\partial^3}{\partial l_1^3}+\cdots
\\
+r l_r\Big(\frac{\kappa}{log(1+\kappa)}\Big)^{r-1}\frac{\partial^{r}}
{\partial l_1^{r}}-n\frac{log(1+\kappa)}{\kappa}\Bigg)=0.
\end{multline}
\end{thm}

\begin{proof}
Assumption \eqref{2.17} is proved based on the operation of expressions \eqref{2.7} and \eqref{2.8} in light of expression \eqref{1.19}.
\end{proof}
	
\section{Symmetric Identities}\label{sec:sym}
Symmetric identities play a crucial role in mathematics and various scientific disciplines due to their inherent properties of symmetry and invariance. These identities describe correlations between variables that hold true across specific transformations, especially when variables are switched around or permuted. Symmetric identities have several applications. They help solve equations and simplify statements in algebra by exposing hidden structures and patterns. Symmetric identities are essential tools in combinatorics that are used to count and analyse permutations, combinations, and arrangements. These identities shed light on the characteristics of symmetric shapes and patterns in geometry. Moreover, symmetric identities have profound implications in physics, often representing conservation laws and fundamental principles in the natural world. Whether in algebraic manipulations, geometric configurations, or physical laws, symmetric identities offer a powerful and elegant framework for understanding and solving problems across diverse mathematical and scientific domains.
We establish symmetric identities for the DMHAP ${}_{\mathbb{H}}\mathbb{A}^{[r]}_n(l_1,l_2,l_3,\cdots,l_r;\kappa)$ in this section by demonstrating the following findings:
	
\begin{thm}
For $\mathcal{I} \neq \mathcal{S}$ with $\mathcal{I} , \mathcal{S}> 0$, we find
\begin{equation}\label{3.1.1.1}
			\mathcal{I}^n {}{}_{\mathbb{H}}\mathbb{A}^{[r]}_n(\mathcal{S}l_1,\mathcal{S}^2 l_2,\mathcal{S}^3 l_3,\cdots,\mathcal{S}^r l_r;\kappa)=\mathcal{S}^n {}{}_{\mathbb{H}}\mathbb{A}^{[r]}_n(\mathcal{I} l_1,\mathcal{I}^2 l_2,\mathcal{I}^3 l_3,\cdots,\mathcal{I}^r l_r;\kappa).
\end{equation}
\end{thm}
	
\begin{proof}
Taking into account that $\mathcal{I} \neq \mathcal{S}$, where $\mathcal{I}, \mathcal{S}> 0$, we proceed as follows:
\begin{equation} \label{3.1.1.2}
\mathbb{R}(\xi;l_1,l_2,l_3,\cdots,l_r;\kappa)=\mathbb{A}(\xi)(1+\kappa)^{\frac{\mathcal{I} \mathcal{S}l_1\xi}{\kappa}}(1 + \kappa)^{\frac{\mathcal{I} ^2 \mathcal{S}^2 l_2 \xi^2}{\kappa}}(1 + \kappa)^{\frac{\mathcal{I}^3 \mathcal{S}^3 l_3 \xi^3}{\kappa}}\cdots(1 + \kappa)^{\frac{\mathcal{I}^r \mathcal{S}^r l_r \xi^r}{\kappa}}.
\end{equation}
Consequently, in terms of $\mathcal{I}$ and $\mathcal{S}$, the formula \eqref{3.1.1.2} for $\mathbb{R}(\xi;l_1,l_2,l_3,\cdots,l_r;\kappa)$ is symmetrical.

Consequently, we can state it as follows:
\begin{multline}\label{3.1.1.3}
\mathbb{R}(\xi;l_1,l_2,l_3,\cdots,l_r;\kappa)= {}{}_{\mathbb{H}}\mathbb{A}^{[r]}_n(\mathcal{I} l_1,\mathcal{I}^2 l_2,\mathcal{I}^3 l_3,\cdots,\mathcal{I}^r l_r;\kappa) \frac{{(\mathcal{S}\xi)}^n}{n!}
\\
=\mathcal{S}^n~{}{}_{\mathbb{H}}\mathbb{A}^{[r]}_n(\mathcal{I} l_1,\mathcal{I}^2 l_2,\mathcal{I}^3 l_3,\cdots,\mathcal{I}^r l_r;\kappa) \frac{{\xi}^n}{n!}.
\end{multline}
Thus, we have
		\begin{multline}\label{3.1.1.4}
				\mathbb{R}(\xi;l_1,l_2,l_3,\cdots,l_r;\kappa)= {}{}_{\mathbb{H}}\mathbb{A}^{[r]}_n(\mathcal{S} l_1,\mathcal{S}^2 l_2,\mathcal{S}^3 l_3,\cdots,\mathcal{S}^r l_r;\kappa) \frac{{(\mathcal{I}\xi)}^n}{n!}
				\\
				=\mathcal{I}^n~{}{}_{\mathbb{H}}\mathbb{A}^{[r]}_n(\mathcal{S} l_1,\mathcal{S}^2 l_2,\mathcal{S}^3 l_3,\cdots,\mathcal{S}^r l_r;\kappa) \frac{{\xi}^n}{n!}.
\end{multline}
The assertion \eqref{3.1.1.1} can be determined by comparing the exponents of the identical powers of $\xi$ in the preceding claims \eqref{3.1.1.3} and \eqref{3.1.1.4}.
	\end{proof}
	
\begin{thm}
For $\mathcal{I} \neq \mathcal{S}$ with $\mathcal{I} , \mathcal{S}> 0$, we find
\begin{multline}\label{3.1.1.5}
\sum_{k=0}^{n}\sum_{m=0}^{k} \binom{n}{k} \binom{k}{m}		\mathcal{I}^k \mathcal{S}^{n+1-k}  
{}{}_{\mathbb{H}}\mathbb{A}^{[r]}_{k-m}(\mathcal{S}l_1,\mathcal{S}^2 l_2,\mathcal{S}^3 l_3,\cdots,\mathcal{S}^r l_r;\kappa) \mathbb{P}_{n-k}(\mathcal{I}-1;\kappa)
\\
=\sum_{k=0}^{n}\sum_{m=0}^{k} \binom{n}{k} \binom{k}{m} \mathcal{S}^k \mathcal{I} ^{n+1-k}  
{}{}_{\mathbb{H}}\mathbb{A}^{[r]}_{k-m}(\mathcal{I}l_1,\mathcal{I}^2 l_2,\mathcal{I}^3 l_3,\cdots,\mathcal{I}^r l_r;\kappa) \mathbb{P}_{n-k}(\mathcal{S}-1;\kappa).
\end{multline}
	\end{thm}
	
	\begin{proof}
Taking into account that $\mathcal{I} \neq \mathcal{S}$, where $\mathcal{I}, \mathcal{S}> 0$, we proceed as follows:
\begin{multline}
\mathbb{S} (\xi;l_1,l_2,l_3,\cdots,l_r;\kappa)=\mathcal{I} \mathcal{S}\xi~\mathbb{A}(\xi) (1+\kappa)^{\frac{\mathcal{I} \mathcal{S}l_1\xi}{\kappa}}
		(1+\kappa)^{\frac{\mathcal{I} ^2 \mathcal{S}^2 l_2\xi^2}{\kappa}}(1+\kappa)^{\frac{\mathcal{I}^3 \mathcal{S}^3 l_3\xi^3}{\kappa}}\cdots
\\
\times~\frac{(1+\kappa)^{\frac{\mathcal{I}^r \mathcal{S}^r l_r\xi^r}{\kappa}}\Bigg((1 + \kappa)^{\frac{\mathcal{I} \mathcal{S}\xi}{\kappa}-1}  \Bigg)}{\left((1 + \kappa)^{\frac{\mathcal{I}  \xi}{\kappa}}-1\right) \left((1 + \kappa)^{\frac{\mathcal{S} \xi}{\kappa}}-1\right) }.
\end{multline}
Assertion \eqref{3.1.1.5} is proven by carrying on in the same way as in the previous theorem.
\end{proof}

\begin{thm}
For $\mathcal{I} \neq \mathcal{S}$ with $\mathcal{I} , \mathcal{S}> 0$, we find
		\begin{equation}\label{3.1.1.6}
			\mathcal{I}^n {}{}_{\mathbb{H}}\mathbb{A}^{[r]}_n(\mathcal{S}l_1,\mathcal{S}^2 l_2,\mathcal{S}^3 l_3,\cdots,\mathcal{S}^r l_r)=\mathcal{S}^n {}{}_{\mathbb{H}}\mathbb{A}^{[r]}_n(\mathcal{I} l_1,\mathcal{I}^2 l_2,\mathcal{I}^3 l_3,\cdots,\mathcal{I}^r l_r).
		\end{equation}
	\end{thm}
	
	\begin{proof}
Taking into account that $\mathcal{I} \neq \mathcal{S}$, where $\mathcal{I}, \mathcal{S}> 0$, we proceed as follows:
\begin{equation} \label{3.1.1.7}
\begin{array}{cc}	
\mathbb{G}(\xi;l_1,l_2,l_3,\cdots,l_r;\kappa)=\mathbb{A}(\xi)(1+\kappa)^{\frac{\mathcal{I} \mathcal{S}l_1\xi}{\kappa}}(1 + \kappa)^{\frac{\mathcal{I} ^2 \mathcal{S}^2 l_2 \xi^2}{\kappa}}(1 + \kappa)^{\frac{\mathcal{I}^3 \mathcal{S}^3 l_3 \xi^3}{\kappa}}\cdots(1 + \kappa)^{\frac{\mathcal{I}^r \mathcal{S}^r l_r \xi^r}{\kappa}}.
\end{array}
\end{equation}
Therefore, in $\mathcal{I}$ and $\mathcal{S}$, the statement \eqref{3.1.1.6} $\mathbb{G}(\xi;l_1,l_2,l_3,\cdots,l_r;\kappa)$ is symmetric.
\end{proof}
	
\begin{thm}
For $\mathcal{I} \neq \mathcal{S}$ with $\mathcal{I} , \mathcal{S}> 0$, we find
\begin{multline}\label{3.1.1.9}
\sum_{k=0}^{n}\sum_{m=0}^{k} \binom{n}{k} \binom{k}{m}	\mathcal{I}^k\mathcal{S}^{n+1-k}\mathcal{B}_n(\kappa){}{}_{\mathbb{H}}\mathbb{A}^{[r]}_{k-m}(\mathcal{S}l_1,\mathcal{S}^2 l_2,\mathcal{S}^3 l_3,\cdots,\mathcal{S}^r l_r;\kappa) \sigma_{n-k}(\mathcal{I}-1;\kappa)
\\
=\sum_{k=0}^{n}\sum_{m=0}^{k} \binom{n}{k} \binom{k}{m} \mathcal{S}^k \mathcal{I} ^{n+1-k} \mathcal{B}_n(\kappa) 
{}{}_{\mathbb{H}}\mathbb{A}^{[r]}_{k-m}(\mathcal{I}l_1,\mathcal{I}^2 l_2,\mathcal{I}^3 l_3,\cdots,\mathcal{I}^r l_r;\kappa) \sigma_{n-k}(\mathcal{S}-1;\kappa).
\end{multline}	
\end{thm}
\begin{proof}
\begin{multline}\label{3.1.1.10}
\mathbb{G}(\xi;l_1,l_2,l_3,\cdots,l_r;\kappa)=\mathcal{I} \mathcal{S}\xi~\mathbb{A}(\xi)(1+\kappa)^{\frac{\mathcal{I} \mathcal{S}l_1\xi}{\kappa}}(1 + \kappa)^{\frac{\mathcal{I} ^2 \mathcal{S}^2 l_2 \xi^2}{\kappa}}(1 + \kappa)^{\frac{\mathcal{I}^3 \mathcal{S}^3 l_3 \xi^3}{\kappa}}\cdots
\\
\times~(1 + \kappa)^{\frac{\mathcal{I}^r \mathcal{S}^r l_r \xi^r}{\kappa}} \frac{\Bigg((1 + \kappa)^{\frac{\mathcal{I} \mathcal{S}\xi}{\kappa}-1}  \Bigg)}{\left((1 + \kappa)^{\frac{\mathcal{I}  \xi}{\kappa}}-1\right) \left((1 + \kappa)^{\frac{\mathcal{S} \xi}{\kappa}}-1\right)}.
\end{multline}	
The expression \eqref{3.1.1.10} mentioned earlier can be expressed as
\begin{align}\label{3.1.1.11}
\mathbb{G}(\xi;l_1,l_2,l_3,\cdots,l_r;\kappa)&=\frac{\mathcal{I} \mathcal{S}\xi}{\left((1 + \kappa)^{\frac{\mathcal{I}  \xi}{\kappa}}-1\right)}\mathbb{A}(\xi)(1+\kappa)^{\frac{\mathcal{I} \mathcal{S}l_1\xi}{\kappa}}(1 + \kappa)^{\frac{\mathcal{I} ^2 \mathcal{S}^2 l_2 \xi^2}{\kappa}}(1 + \kappa)^{\frac{\mathcal{I}^3 \mathcal{S}^3 l_3 \xi^3}{\kappa}}
\nonumber \\
&\qquad\qquad\qquad\qquad
\times~\cdots(1 + \kappa)^{\frac{\mathcal{I}^r \mathcal{S}^r l_r \xi^r}{\kappa}}\frac{\Bigg((1 + \kappa)^{\frac{\mathcal{I} \mathcal{S}\xi}{\kappa}-1}  \Bigg)} {\left((1 + \kappa)^{\frac{\mathcal{S} \xi}{\kappa}}-1\right)}
\nonumber \\
&=\mathcal{S}\sum\limits_{m=0}^{\infty} \mathcal{B}_m(\kappa)\frac{(\mathcal{I}\xi)^m}{m!}\sum\limits_{k=0}^{\infty}{}{}_{\mathbb{H}}\mathbb{A}^{[r]}_{k}(\mathcal{S}l_1,\mathcal{S}^2 l_2,\mathcal{S}^3 l_3,\cdots,\mathcal{S}^r l_r;\kappa)\frac{(\mathcal{I}\xi)^k}{k!} 
\\
&\qquad\qquad\qquad\qquad\qquad\qquad\qquad
\times\sum\limits_{n=0}^{\infty}  \sigma_{n}(\mathcal{I}-1;\kappa)\frac{(\mathcal{S}\xi)^n}{n!}
\nonumber \\
&=\sum\limits_{n=0}^{\infty}\Bigg[\sum_{k=0}^{n}\sum_{m=0}^{k} \binom{n}{k}\binom{k}{m}		\mathcal{I}^k\mathcal{S}^{n+1-k}\mathcal{B}_n(\kappa){}{}_{\mathbb{H}}\mathbb{A}^{[r]}_{k-m}(\mathcal{S}l_1,\mathcal{S}^2 l_2,\mathcal{S}^3 l_3,\cdots,
\nonumber \\
&\qquad\qquad\qquad\qquad\qquad\qquad \mathcal{S}^r l_r;\kappa)~\sigma_{n-k}(\mathcal{I}-1;\kappa)\Bigg]\frac{\xi^n}{n!}.
\end{align}
Similar to that, we may write
\begin{multline*} 
\mathbb{G}(\xi;l_1,l_2,l_3,\cdots,l_r;\kappa)=\frac{\mathcal{I} \mathcal{S}\xi}{\left((1 + \kappa)^{\frac{\mathcal{S}  \xi}{\kappa}}-1\right)}\mathbb{A}(\xi)(1+\kappa)^{\frac{\mathcal{I} \mathcal{S}l_1\xi}{\kappa}}(1 + \kappa)^{\frac{\mathcal{I} ^2 \mathcal{S}^2 l_2 \xi^2}{\kappa}}(1 + \kappa)^{\frac{\mathcal{I}^3 \mathcal{S}^3 l_3 \xi^3}{\kappa}}
\\
\times~\cdots(1 + \kappa)^{\frac{\mathcal{I}^r \mathcal{S}^r l_r \xi^r}{\kappa}}\frac{\Bigg((1 + \kappa)^{\frac{\mathcal{I} \mathcal{S}\xi}{\kappa}-1}  \Bigg)} {\left((1 + \kappa)^{\frac{\mathcal{I} \xi}{\kappa}}-1\right)}
\end{multline*}				
\begin{multline}\label{3.1.1.12}
=\mathcal{I}\sum\limits_{m=0}^{\infty} \mathcal{B}_m(\kappa)\frac{(\mathcal{S}\xi)^m}{m!}\sum\limits_{k=0}^{\infty}{}{}_{\mathbb{H}}\mathbb{A}^{[r]}_{k}(\mathcal{I}l_1,\mathcal{I}^2 l_2,\mathcal{I}^3 l_3,\cdots,\mathcal{I}^r l_r;\kappa)\frac{(\mathcal{S}\xi)^k}{k!} \sum\limits_{n=0}^{\infty}  
\sigma_{n}(\mathcal{S}-1;\kappa)\frac{(\mathcal{I}\xi)^n}{n!}
\\
=\sum\limits_{n=0}^{\infty}\Bigg[\sum_{k=0}^{n}\sum_{m=0}^{k} \binom{n}{k}\binom{k}{m}		\mathcal{S}^k\mathcal{I}^{n+1-k}\mathcal{B}_n(\kappa){}{}_{\mathbb{H}}\mathbb{A}^{[r]}_{k-m}(\mathcal{I}l_1,\mathcal{I}^2 l_2,\mathcal{I}^3 l_3,\cdots,\mathcal{I}^r l_r;\kappa)
\\
\times~ \sigma_{n-k}(\mathcal{S}-1;\kappa)\Bigg]\frac{\xi^n}{n!}.
\end{multline}
Assertion \eqref{3.1.1.9} is established by comparing the identical exponents of expressions \eqref{3.1.1.11} and \eqref{3.1.1.12}.
\end{proof}

\section{Operational formalism and applications}\label{sec:op}
Especially in the case of doped-type polynomials, operational strategies have shown to be quite successful in producing novel sets of polynomials. Using these techniques, a starting parental polynomial is subjected to certain differential or difference operators, creating a new polynomial with associated characteristics. Using these operators, families of orthogonal polynomials may be produced that, although related to classical orthogonal polynomials, have extra symmetries. In umbral calculus, where a new variable meeting particular algebraic characteristics is introduced, operational techniques also play a major role. With this introduction, new polynomial families with additional features connected to classical polynomials may be constructed. 

Operational approaches, therefore, have vital applications in various mathematical and physical settings and are effective instruments for creating new families of polynomials with attributes similar to the parental polynomial.

By repeatedly differentiating~\eqref{2.2} with respect to $l_1,l_2,l_3,\cdots,l_r$, we get at
\begin{equation}\label{4.1}
\frac{log(1+\kappa)}{\kappa}~n~
\Big\{ {}{}_{\mathbb{H}}\mathbb{A}^{[r]}_{n-1}(l_1,l_2,l_3,\cdots,l_r;\kappa)\Big\}=D_{l_1}\Big\{ {}{}_{\mathbb{H}}\mathbb{A}^{[r]}_n(l_1,l_2,l_3,\cdots,l_r;\kappa)\Big\}
;
\end{equation}
\begin{equation}\label{4.2}
\Big(\frac{log(1+\kappa)}{\kappa}\Big)^2~n(n-1)~\Big\{{}{}_{\mathbb{H}}\mathbb{A}^{[r]}_{n-1}(l_1,l_2,l_3,\cdots,l_r;\kappa)\Big\}=D^2_{l_1}\Big\{{}{}_{\mathbb{H}}\mathbb{A}^{[r]}_n(l_1,l_2,l_3,\cdots,l_r;\kappa)\Big\};
\end{equation}
\begin{equation}\label{4.3}
\Big(\frac{log(1+\kappa)}{\kappa}\Big)^3~n(n-1)(n-2)~\Big\{{}{}_{\mathbb{H}}\mathbb{A}^{[r]}_{n-1}(l_1,l_2,l_3,\cdots,l_r;\kappa)\Big\}=D^3_{l_1}\Big\{{}{}_{\mathbb{H}}\mathbb{A}^{[r]}_n(l_1,l_2,l_3,\cdots,l_r;\kappa)\Big\};
\end{equation}
\begin{equation*}
\vdots~\hspace{4cm}{~~}~\vdots
\end{equation*}
\begin{multline}\label{4.4}
\Big(\frac{log(1+\kappa)}{\kappa}\Big)^r~n(n-1)(n-2)\cdots(n-r+1)~
\times~\Big\{{}{}_{\mathbb{H}}\mathbb{A}^{[r]}_{n-1}(l_1,l_2,l_3,\cdots,l_r;\kappa)\Big\}=\\
D^r_{l_1}\Big\{{}{}_{\mathbb{H}}\mathbb{A}^{[r]}_n(l_1,l_2,l_3,\cdots,l_r;\kappa)\Big\}.
\end{multline}
	
Also,
\begin{equation}\label{4.5}
\frac{log(1+\kappa)}{\kappa}~n(n-1)~\Big\{ {}{}_{\mathbb{H}}\mathbb{A}^{[r]}_{n-1}(l_1,l_2,l_3,\cdots,l_r;\kappa)\Big\}=D_{l_2}\Big\{ {}{}_{\mathbb{H}}\mathbb{A}^{[r]}_{n}(l_1,l_2,l_3,\cdots,l_r;\kappa)\Big\}
;
\end{equation}
\begin{equation}\label{4.6}
\frac{log(1+\kappa)}{\kappa}~n(n-1)(n-2)~\Big\{ {}{}_{\mathbb{H}}\mathbb{A}^{[r]}_{n-1}(l_1,l_2,l_3,\cdots,l_r;\kappa)\Big\}=D_{l_3}\Big\{ {}{}_{\mathbb{H}}\mathbb{A}^{[r]}_{n}(l_1,l_2,l_3,\cdots,l_r;\kappa)\Big\}
;
\end{equation}
\begin{equation*}
\vdots~\hspace{4cm}{~~}~\vdots
\end{equation*}
\begin{multline}\label{4.7}
\frac{log(1+\kappa)}{\kappa}~n(n-1)
(n-2)\cdots(n-r+1)\Big\{ {}{}_{\mathbb{H}}\mathbb{A}^{[r]}_{n-1}(l_1,l_2,l_3,\cdots,l_r;\kappa)\Big\}=\\
D_{l_r}\Big\{ {}{}_{\mathbb{H}}\mathbb{A}^{[r]}_{n}(l_1,l_2,l_3,\cdots,l_r;\kappa)\Big\},
\end{multline}
correspondingly.

Considering the results of equations~\eqref{4.1}--\eqref{4.7}, the subsequent equations' solutions are ${}{}_{\mathbb{H}}\mathbb{A}^{[r]}_{n-1}(l_1,l_2,l_3,\cdots,l_r;\kappa)$:
\begin{equation}\label{4.8}
		\frac{\kappa}{log(1+\kappa)}~D^2_{l_1}\Big\{{}{}_{\mathbb{H}}\mathbb{A}^{[r]}_{n-1}(l_1,l_2,l_3,\cdots,l_r;\kappa)\Big\}
		=D_{l_2}\Big\{{}{}_{\mathbb{H}}\mathbb{A}^{[r]}_{n-1}(l_1,l_2,l_3,\cdots,l_r;\kappa)\Big\};\\
	\end{equation}
    
	\begin{equation}\label{4.9}
		\Big(\frac{\kappa}{log(1+\kappa)}\Big)^2~D^3_{l_1}\Big\{ {}{}_{\mathbb{H}}\mathbb{A}^{[r]}_{n-1}(l_1,l_2,l_3,\cdots,l_r;\kappa)\Big\}
		=D_{l_3}\Big\{{}{}_{\mathbb{H}}\mathbb{A}^{[r]}_{n-1}(l_1,l_2,l_3,\cdots,l_r;\kappa)\Big\};\\
	\end{equation}
    
	\begin{equation*}
		\vdots~\hspace{4cm}{~~}~\vdots
	\end{equation*}
\begin{equation}\label{4.10}
\Big(\frac{\kappa}{log(1+\kappa)}\Big)^{r-1}~D^r_{l_1}\Big\{ {}{}_{\mathbb{H}}\mathbb{A}^{[r]}_{n-1}(l_1,l_2,l_3,\cdots,l_r;\kappa)\Big\}
=D_{l_r}\Big\{{}{}_{\mathbb{H}}\mathbb{A}^{[r]}_{n-1}(l_1,l_2,l_3,\cdots,l_r;\kappa)\Big\},\\
\end{equation}

respectively, Considering the provided initial conditions, we have the following
\begin{equation}\label{4.11}
{}_{\mathbb{H}}\mathbb{A}^{[r]}_{n}(l_1,0,0,\cdots,0;0)={}\mathbb{A}_{n}(l_1).\\
\end{equation}

Thus, considering expressions~\eqref{4.8}--\eqref{4.11}, it can be concluded that
\begin{multline}\label{4.12}
{}{}_{\mathbb{H}}\mathbb{A}^{[r]}_{n}(l_1,l_2,l_3,\cdots,l_r;\kappa)=\exp\Bigg(\frac{l_2~\kappa}{log(1+\kappa)}~D^2_{l_1}
+l_3\Big(\frac{\kappa}{log(1+\kappa)}\Big)^2~D^3_{l_1}+\cdots
\\
+l_r\Big(\frac{\kappa}{log(1+\kappa)}\Big)^{r-1}~D^r_{l_1}\Bigg)\{{}\mathbb{H}_{n}(l_1;\kappa)\}.\\
\end{multline}
	
From the previously indicated viewpoint, the polynomials ${}{}_{\mathbb{H}}\mathbb{A}^{[r]}_{n}(l_1,l_2,l_3,\cdots,l_r;\kappa)$ may be obtained from the degenerate polynomials $\mathbb{A}_n(l_1)$ using the operational rule \eqref{4.12}.

Choosing an appropriate function $\mathbb{A}(l_1)$ makes it possible to produce various Appell polynomial members. These members are included in Table 1, along with their names, series definitions, generating functions, and associated numbers.\\

\textbf{Table 1.~``Certain number of the Appell family members".} \\
\\
{\tiny{
\begin{tabular}{llllll}
\hline
&&&&\\
{\bf S.}  & {\bf Name of the} & $\mathbb{A}(\xi)$ & {\bf Generating expression} &  {\bf Series representation}\\
{\bf No.}&{\bf polynomials and }&&&\\
&{\bf related numbers }&&&\\
&&&&\\
			\hline
			&&&&\\
	{\bf I.} & Bernoulli &   $\left(\frac{\xi}{e^{\xi}-1}\right)$  & $\left(\frac{\xi}{e^{\xi}-1}\right) e^{\xi {l_1}}=\sum\limits_{
				k=0}^\infty \mathcal{S}_{k}(l_1)\frac{{\xi}^k}{k!}$  & $\mathcal{S}_{k}(l_1)=\sum\limits_{m=0}^k {k \choose m} \mathcal{S}_{m} \xi^{k-m}$\\
			& polynomials &&$\left(\frac{\xi}{e^{\xi}-1}\right)=\sum\limits_{
				k=0}^\infty \mathcal{S}_{k}\frac{{\xi}^k}{k!}$ &\\
			& and numbers \cite{ERDEYLHTF}&&$\mathcal{S}_{k}:=\mathcal{S}_{k}(0)$&\\
			\hline
			&&&&\\
			{\bf II.} & Euler&   $\left(\frac{2}{e^{\xi}+1}\right)$  & $\left(\frac{2}{e^{\xi}+1}\right)e^{l_1 \xi}=\sum\limits_{
				k=0}^\infty \mathfrak{E}_{k}(l_1)\frac{\xi^k}{k!}$  & $\mathfrak{E}_{k}(l_1)=\sum\limits_{m=0}^{k}{k\choose m}\frac{\mathfrak{E}_{m}}{2^{m}}{\left(l_1-\frac{1}{2}\right)}^{k-m}$\\
			& polynomials &&$\frac{2e^{\xi}}{e^{2{\xi}}+1}=\sum\limits_{k=0}^{\infty}\mathfrak{E}_{k}\frac{{\xi}^{k}}{k!}$ &\\
			& and numbers \cite{ERDEYLHTF}&&$\mathfrak{E}_{k}:=2^{k}\mathfrak{E}_{k}\left(\frac{1}{2}\right)$&\\
			\hline
			&&&&\\
			{\bf III.} & Genocchi&   $\left(\frac{2{\xi}}{e^{\xi}+1}\right)$  & $\left(\frac{2{\xi}}{e^{\xi}+1}\right)e^{l_1 \xi}=\sum\limits_{
				k=0}^\infty \mathfrak{G}_{k}(l_1)\frac{{\xi}^k}{k!}$  & $\mathfrak{G}_{k}(l_1)=\sum\limits_{m=0}^{k}{k\choose m}\mathfrak{G}_m~ {l_1}^{k-m}$\\
			& polynomials &&$\frac{2{\xi}}{e^{\xi}+1}=\sum\limits_{k=0}^{\infty}\mathfrak{G}_{k}\frac{{\xi}^{k}}{k!}$ &\\
			& and numbers \cite{SandorHNT}&&$\mathfrak{G}_{k}:=\mathfrak{G}_k(0)$&\\
			\hline
\end{tabular}}}
\vspace{.25cm}\\

The ``Bernoulli, Euler, and Genocchi numbers" are important in many different areas of mathematics, demonstrating their applicability and importance. Numerous mathematical fields employ these integers in various ways. Examples of these numbers used in mathematical formulas include the ``Bernoulli polynomials and the Euler-Maclaurin formula". The basic elements of these formulas are the Bernoulli numbers, a series of rational integers.

These polynomials have significant consequences in various domains, including number theory, combinatorics, and numerical analysis, in addition to their direct applications. Their use goes beyond simple computations since they form important links with algebraic geometry and representation theory, among other areas of mathematics. 

Another series of integers in several mathematical domains is known as the ``Euler numbers". They arise in number theory, combinatorics, elliptic curve theory, and the study of algebraic topology and smooth manifold geometry. ``Euler numbers" are important for modular forms research as well, with applications in coding theory and cryptography.

A series of integers known as the Genocchi numbers may be found in several combinatorial tasks, including labelled rooted trees and counting up-down sequences. They have uses in automata theory and graph theory and are also connected to the Riemann zeta function.

Functions that are closely associated with the ``Euler numbers" are the hyperbolic secant function and the trigonometric function. The ``Euler numbers" and their derivatives are included in the Taylor series representations, which are used in a variety of mathematical and physical domains, including quantum field theory and signal processing.

Because of their numerous linkages to other fields of mathematics and their many practical applications, these numbers are intriguing to study.

By treating ``Bernoulli, Euler, and Genocchi polynomials" as members of the Appell family, we may get different members of the multidimensional Hermite-based Appell polynomials. This involves the construction of ``multidimensional Hermite-based Bernoulli polynomials"
 ${}{}_{\mathbb{H}}\mathbb{B}^{[r]}_n(l_1,l_2,l_3,\cdots,l_r;\kappa)$, multidimensional Hermite based Euler polynomials ${}{}_{\mathbb{H}}\mathbb{E}^{[r]}_n(l_1,l_2,l_3,\cdots,l_r;\kappa)$ and multidimensional Hermite based Genocchi polynomials ${}{}_{\mathbb{H}}\mathbb{G}^{[r]}_n(l_1,l_2,l_3,\cdots,l_r;\kappa)$. The ``Bernoulli, Euler, and Genocchi polynomials" are among the several varieties of polynomials that are members of the multidimensional Appell polynomial family, which is based on Hermite. These polynomials may all be produced using a particular selection of the function $\mathbb{A}(l_1)$; the generating equations for these polynomials are given below:
\begin{equation}
\frac{\xi}{(e^{\xi}-1)}		(1+\kappa)^{\xi\Big(\frac{l_1+l_2\xi+l_3\xi^2+\cdots+l_n\xi^{n-1}}{\kappa}\Big)}	=\sum_{n=0}^{\infty} {}_{\mathbb{H}}\mathbb{B}^{[r]}_n(l_1,l_2,l_3,\cdots,l_r;\kappa)\frac{\xi^n}{n!},\\
\end{equation}

\begin{equation}
\frac{2}{(e^{\xi}+1)}		(1+\kappa)^{\xi\Big(\frac{l_1+l_2\xi+l_3\xi^2+\cdots+l_n\xi^{n-1}}{\kappa}\Big)}	=\sum_{n=0}^{\infty} {}_{\mathbb{H}}\mathbb{E}^{[r]}_n(l_1,l_2,l_3,\cdots,l_r;\kappa)\frac{\xi^n}{n!}\\
\end{equation}

and
\begin{equation}
\frac{2\xi}{(e^{\xi}+1)}		(1+\kappa)^{\xi\Big(\frac{l_1+l_2\xi+l_3\xi^2+\cdots+l_n\xi^{n-1}}{\kappa}\Big)}	=\sum_{n=0}^{\infty} {}_{\mathbb{H}}\mathbb{G}^{[r]}_n(l_1,l_2,l_3,\cdots,l_r;\kappa)\frac{\xi^n}{n!},\\
\end{equation}
respectively.\\

As a consequence, we are able to prove the corresponding findings for the ``multidimensional Bernoulli, Euler, and Genocchi polynomials based on Hermite ones".

\section{Conclusion}\label{sec:con}
The operational methods discussed find widespread application across numerous branches of mathematical physics, including quantum mechanics and classical optics, enabling the comprehension of intricate physical systems through manageable operators. Here, we introduce the degenerate multidimensional Hermite-based Appell polynomials, highlighting their unique characteristics. Establishing a generating link and delineating quasi-monomial properties are presented alongside symmetric identities and operational rules. These polynomials, integral to quantum physics and beyond, offer a versatile framework for mathematical analysis, defined by initial conditions and recurrence relations.

Hermite and Appell polynomials of degenerate forms are invaluable tools in mathematics and science, offering completeness for representing square-integrable functions and efficient computation via three-term recurrence relations. Their utility spans various domains, including quantum mechanics, approximation theory, numerical analysis, and the study of random processes like Brownian motion. Moreover, they find applications in biological and medical sciences, aiding signal analysis. The exploration of operational formalism not only facilitates the discovery of extended polynomial versions but offers insights into differential equations, integral equations, recurrence relations, and other mathematical structures. This approach catalyzes the development of novel mathematical tools and methodologies, fostering a deeper understanding of diverse mathematical phenomena.

Future work on degenerate multidimensional Hermite-based Appell polynomials (DMHAP) could focus on extending them to broader families of special functions by incorporating additional parameters, weight functions, or constraints to address diverse problems in applied mathematics, physics, and engineering. Developing efficient computational algorithms for evaluating DMHAP will be critical for enabling numerical simulations in higher dimensions and degenerate cases, enhancing their practical applicability. The operational principles and orthogonal properties of DMHAP hold promise for solving higher-dimensional partial differential equations, studying quantum mechanics problems, and modeling wave phenomena. Further exploration of the connections between DMHAP and other polynomial families, such as Laguerre, Legendre, and Chebyshev polynomials, could reveal deeper theoretical insights and unify them within the broader field of special functions. Additionally, the symmetric identities derived in this study could lead to new applications in combinatorics, generating functions, and representation theory, offering a rich avenue for future research.





\section*{Acknowledgements}
We sincerely thank the reviewers for their valuable feedback and constructive suggestions, which have significantly enhanced the clarity and quality of our manuscript. Their insightful comments have guided us in refining the presentation and addressing key aspects of the study. We deeply appreciate their time and effort in reviewing our work, which has contributed greatly to its overall improvement.


\end{document}